\documentclass[12pt,reqno]{amsart}

\usepackage[T1]{fontenc}
\usepackage[utf8]{inputenc}
\usepackage[colorlinks=true]{hyperref}
\hypersetup{urlcolor=blue, linkcolor=blue, citecolor=red, anchorcolor=blue]}
\usepackage{amsmath,amsfonts,amssymb,amsxtra,bbm,paralist}
\usepackage{lmodern}
\usepackage{setspace,graphicx,color,pdflscape}
\usepackage{dsfont}
\usepackage{bookmark}

\newcounter{Hequation}

\makeatletter\g@addto@macro\equation{\stepcounter{Hequation}}\makeatother

\newtheorem{theorem}{Theorem}
\newtheorem{proposition}[theorem]{Proposition}
\newtheorem{lemma}[theorem]{Lemma}
\newtheorem{corollary}[theorem]{Corollary}

\def\sfc{\mathsf c}
\def\sfA{\mathsf A}
\def\sfB{\mathsf B}
\def\sfC{\mathsf C}
\def\sfH{\mathsf H}
\def\sfL{\mathsf L}
\def\sfPi{\mathsf\Pi}
\def\sfT{\mathsf T}

\newcommand{\Email}[1]{\email{\href{mailto:#1}{\textsf{#1}}}}
\newcommand{\step}[1]{\medskip\noindent\textit{\textbf{$\bullet$ #1}}}
\newcommand{\be}[1]{\begin{equation}\label{#1}}
\newcommand{\ee}{\end{equation}}
\renewcommand{\(}{\left(}
\renewcommand{\)}{\right)}
\newcommand{\C}{{\mathcal C}}
\newcommand{\R}{{\mathbb R}}

\newcommand{\intw}[2]{\int_{\R^d}{#1}\dd #2}
\newcommand{\iintw}[2]{\iint_{\R^d\times\R^d}{#1}\dd #2}
\newcommand{\n}[1]{\left\|#1\right\|}
\newcommand{\nrm}[2]{\n{#1}_{#2}}
\newcommand{\bangle}[1]{\left\langle #1\right\rangle}
\newcommand{\scalar}[2]{\left\langle{#1},{#2}\right\rangle}
\newcommand{\op}[1]{\mathsf{#1}}
\renewcommand\d{\mathrm d}
\newcommand{\rint}{\int_{\R^d}}
\newcommand{\riint}{\iint_{\R^d\times\R^d}}
\newcommand\Dx{\nabla_x}
\newcommand\Dv{\nabla_v}
\newcommand\dd{\,\d}
\newcommand\dt{\frac{\mathrm d}{\mathrm dt}}

\newcommand{\init}{\mathrm{in}} 
\newcommand{\M}{F} 
\newcommand{\gam}{\alpha} 
\newcommand{\cst}{\delta} 
\newcommand{\bb}{\mathrm b} 
\newcommand{\dmu}{\d\mu} 
\newcommand{\ddmu}{\dd\mu} 
\newcommand{\Id}{\mathsf{Id}} 

\newcommand{\dnu}{\d\nu} 

\begin{document}

\title[Sub-exponential Hypocoercivity]{Hypocoercivity and sub-exponential local equilibria}

\date{\today}

\author[E.~Bouin, J.~Dolbeault, L.~Lafleche, C.~Schmeiser]{E.~Bouin, J.~Dolbeault, L.~Lafleche, C.~Schmeiser}

\address[Emeric Bouin]{CEREMADE (CNRS UMR n$^\circ$ 7534), PSL university, Universit\'e Paris-Dauphine, Place de Lattre de Tassigny, 75775 Paris 16, France\hfill\ }
\Email{bouin@ceremade.dauphine.fr}

\address[Jean Dolbeault]{CEREMADE (CNRS UMR n$^\circ$ 7534), PSL university, Universit\'e Paris-Dauphine, Place de Lattre de Tassigny, 75775 Paris 16, France\hfill\ }
\Email{dolbeaul@ceremade.dauphine.fr}

\address[Laurent Lafleche]{CEREMADE (CNRS UMR n$^\circ$ 7534), PSL university, Universit\'e Paris-Dauphine, Place de Lattre de Tassigny, 75775 Paris 16, France\hfill\ }
\Email{lafleche@ceremade.dauphine.fr}

\address[Christian Schmeiser]{Fakult\"at f\"ur Mathematik, Universit\"at Wien, Oskar-Morgenstern-Platz~1, 1090 Wien, Austria\hfill\ }
\Email{Christian.Schmeiser@univie.ac.at}


\begin{abstract} Hypocoercivity methods are applied to linear kinetic equations without any space confinement, when local equilibria have a sub-exponential decay. By Nash type estimates, global rates of decay are obtained, which reflect the behavior of the heat equation obtained in the diffusion limit. The method applies to Fokker-Planck and scattering collision operators. The main tools are a weighted Poincar\'e inequality (in the Fokker-Planck case) and norms with various weights. The advantage of weighted Poincar\'e inequalities compared to the more classical weak Poincar\'e inequalities is that the description of the convergence rates to the local equilibrium does not require extra regularity assumptions to cover the transition from super-exponential and exponential local equilibria to sub-exponential local equilibria.\end{abstract}



\maketitle\thispagestyle{empty}\vspace*{-1cm}


\section{Introduction}\label{sec:introduction}

This paper is devoted to a \emph{hypocoercivity} method designed for obtaining decay rates in weighted $\mathrm L^2$ norms of the solution to the Cauchy problem
\be{eq:main}
\begin{cases}
\partial_tf+v\cdot\Dx f=\sfL f\,,\\[4pt]
f(0,x,v)=f^\init(x,v)\,,
\end{cases}
\ee
for a distribution function $f(t,x,v)$, with \emph{position} $x\in\R^d$, \emph{velocity} $v\in\R^d$, and \emph{time} \hbox{$t\ge0$}. The linear \emph{collision operator} $\sfL$ acts only on the velocity variable and its null space is assumed 
to be one-dimensional and spanned by the \emph{local equilibrium} $\M$, a probability density of the form
\be{SubExponential}
   \M(v)=C_\gam\,e^{-\bangle v^\gam} \,,\quad v\in\R^d \,,\qquad\mbox{with } C_\gam^{-1} = \rint e^{-\bangle v^\gam} \dd v \,,
\ee
where we use the notation
\[
\bangle v:=\sqrt{1+|v|^2}\,.
\]
Our results will be concerned with the sub-exponential case $0<\alpha<1$, as opposed to the exponential ($\alpha=1$)
and super-exponential ($\alpha> 1$, including the Gaussian with $\alpha=2$) cases. This specific choice of the form of the equilibrium is for notational convenience in the proofs. The results can easily be extended to more general distributions~$F$, satisfying
\[
   \alpha := \lim_{|v|\to+\infty} \frac{\log(-\log F(v))}{\log |v|} \in (0,1) \,.
\]

We shall consider two types of collision operators, either the \emph{Fokker-Planck} operator
\[
\sfL_1 f=\Dv\cdot\Big(\M\,\Dv\big(\M^{-1}\,f\big)\Big) \,,
\]
or the \emph{scattering} operator
\[
\sfL_2 f=\rint\bb(\cdot,v')\,\Big(f(v')\,\M(\cdot)-f(\cdot)\,\M(v')\Big)\dd v'\,.
\]
We assume \emph{local mass conservation}
\[
\rint\sfL f\dd v=0\,,
\]
which always holds for $\sfL=\sfL_1$, and also for $\sfL=\sfL_2$ under the assumption
\be{hyp:b_mass}\tag{H1}
\rint\big(\bb(v,v')-\bb(v',v)\big)\,\M(v')\dd v'=0  \,.
\ee
Note that \emph{micro-reversibility}, \emph{i.e.}, the symmetry of $\bb$, is not required. 

Further assumptions on the \emph{cross-section} $\bb$ that will be given below guarantee that the operators $\sfL_1$ and 
$\sfL_2$ are responsible for the same type of asymptotic behavior. As a motivation, the relaxation properties of $\sfL_1$ 
can be made transparent by the symmetrizing transformation $f=g\,\sqrt{F}$, leading to the transformed operator
\[
g\mapsto\frac{1}{\sqrt{F}}\,\nabla_v\cdot \left(F\,\nabla_v\,\frac{g}{\sqrt{F}}\right) = \Delta_v g - \nu_1(v)\,g
\]
with the \emph{collision frequency}
\be{nu1}
  \nu_1(v) = \frac{\Delta_v F}{2F} - \frac{|\nabla_v F|^2}{4F^2} \approx \frac{\alpha^2}{4} |v|^{-2(1-\alpha)}
  \qquad\mbox{as } |v|\to\infty \,.
\ee
Partially motivated by this, we assume the existence of constants $\beta$, $\overline\bb$, $\underline\bb>0$, $\gamma\ge 0$, with 
$\gamma\le\beta$,  $\gamma<d$, such that
\be{hyp:b_bounds}\tag{H2}
   \underline\bb \bangle{v}^{-\beta} \bangle{v'}^{-\beta} \le \bb(v,v') \le \overline\bb \min\big\{|v-v'|^{-\beta}, |v-v'|^{-\gamma}\big\} \,.
\ee
The upper bound with the restriction on the exponent $\gamma$ serves as a local integrability assumption. Hypotheses
\eqref{hyp:b_mass} and~\eqref{hyp:b_bounds} permit the choice $\bb(v,v') = \bangle{v}^{-\beta}\bangle{v'}^{-\beta}$ with
arbitrary $\beta>0$, as well as Boltzmann kernels $\bb(v,v') = |v-v'|^{-\beta}$ with $0<\beta<d$.

As a consequence of~\eqref{hyp:b_bounds} the collision frequency
\[
   \nu_2(v) = \rint\bb(v,v')\,\M(v')\dd v'
\]
satisfies
\begin{align*}
  \underline\bb \bangle{v}^{-\beta} \rint \bangle{v'}^{-\beta}F(v')\dd v' \le \nu_2(v)
  \le &\,\,\overline\bb \int_{|v-v'|<1} |v-v'|^{-\gamma}F(v')\dd v' \\ &+  \overline\bb \int_{|v-v'|>1} |v-v'|^{-\beta}F(v')\dd v' \,.
\end{align*}
It is obvious that the last term is $O(|v|^{-\beta})$ as $|v|\to\infty$, and the first term on the right hand side is asymptotically 
small compared to that as a consequence of the sub-exponential decay of $F$. Therefore there exist constants 
$\overline\nu\ge \underline\nu>0$ such that 
\be{nu2-est}
   \underline\nu \bangle{v}^{-\beta} \le \nu_2(v) \le \overline\nu \bangle{v}^{-\beta} \qquad\forall\, v\in\R^d \,, 
\ee
and the behavior for large $|v|$ is as in~\eqref{nu1} with $\beta=2(1-\alpha)$.

Since both collision operators are propagators of Markov processes with the same positive stationary distribution $F$, they
also share the (quadratic) entropy dissipation property
\[
  \frac{1}{2} \frac{\dd}{\dd t} \iintw{f^2} x\dd\mu= \iintw{(\sfL f)f} x\dd\mu \le 0 \,,\qquad\mbox{with } \dd\mu(v) := \frac{\dd v}{F(v)} \,,
\]
where the dissipations are given by
\be{L1-diss}
  - \rint (\sfL_1 f)f\,\dd\mu = \rint \left| \nabla_v \frac{f}{F}\right|^2 F\,\dd v
\ee
and
\be{L2-diss}
  - \rint (\sfL_2 f)f\,\dd\mu = \frac{1}{2} \rint\rint \bb(v,v') \big( f'F - fF'\big)^2 \dd\mu \dd\mu' \,,
\ee
with the prime denoting evaluation at $v'$. For a derivation of~\eqref{L2-diss} see, \emph{e.g.},~\cite{MR1803225,MR2763032}.

Our purpose is to consider solutions of~\eqref{eq:main} with non-negative initial datum $f^\init$ and to study their large time behavior. If $f^\init$ has finite mass, then mass is conserved for any $t\ge0$. Since there is no stationary state with finite mass, it is expected that $f(t,\cdot,\cdot)$ locally tends to zero as $t\to+\infty$. However, the dissipations~\eqref{L1-diss} and 
\eqref{L2-diss} vanish for arbitrary \emph{local equilibria} of the form $f(t,x,v) = \rho(t,x)\,F(v)$, and therefore the analysis of
the decay to zero requires an \emph{hypocoercivity} method. 

For the formulation of our main result, we introduce the norms
\be{Norms}
\|f\|_k :=\(\iintw{f^2\,\bangle v^k}{x\dd\mu}\)^{1/2} \,,\qquad k\in\R,
\ee
as well as the scalar product $\scalar{f_1}{f_2}:=\riint f_1\,f_2\dd x\ddmu$ on $\mathrm L^2(\d x\ddmu)$ with the induced
norm $\|f\|^2:=\|f\|_0^2 = \scalar ff$.
\begin{theorem}\label{th:main} Let $\gam\in(0,1)$, $\beta>0$, $k>0$ and let $\M$ be given by~\eqref{SubExponential}. Assume that either $\sfL=\sfL_1$ and $\beta=2\,(1-\gam)$, or $\sfL=\sfL_2$ and~\eqref{hyp:b_mass},~\eqref{hyp:b_bounds}. Then there exists a constant $\C>0$ such that any solution $f$ of~\eqref{eq:main} with initial datum $f^\init\in\mathrm L^2(\bangle v^k\d x\ddmu)\cap\mathrm L^1_+(\d x\dd v)$ satisfies
\[
   \|f(t,\cdot,\cdot)\|^2 \le \C\,\frac{\nrm{f^\init}{}^2}{(1+\kappa\,t)^{\,\zeta}} \qquad t\ge 0
\]
with rate $\zeta=\min\left\{d/2,k/\beta\right\}$ and with $\kappa>0$, which is an explicit function of the two quotients $\nrm{f^\init}{}/\nrm{f^\init}k$ and $\nrm{f^\init}{}/\nrm{f^\init}{\mathrm L^1(\d x\dd v)}$.\end{theorem}
The proof relies on the $\mathrm L^2$-hypocoercivity approach of~\cite{MR2576899,MR3324910}. An important ingredient is 
\emph{microscopic coercivity,} meaning that the entropy dissipation controls the distance to the set of local equilibria. For 
$\sfL=\sfL_1$ and for the exponential and super-exponential cases $\gam\ge1$, this control is provided by the \emph{Poincar\'e inequality}
\be{Poincare}
  \rint \left| \nabla_v g\right|^2 F\,\dd v \ge  \C_P \rint \left(g - \overline{g}\right)^2 F\,\dd v \,,
\ee
with $\overline{g} = \rint gF\dd v$ and $\C_P>0$ implying, with $g=f/F$,
\[
   -\bangle{\sfL_1 f,f} \ge \C_P\,\|f-\rho_f F\|^2 \,,
\] 
with $\rho_f = \rint f\,dv$.
For this case the result of Theorem~\ref{th:main} (with $k=0$) has been proven in~\cite{hal-01575501}. For the sub-exponential case
of this work we shall prove a relaxed version.
\begin{lemma}\label{lem:micro-coerc}
Let $F$ be given by~\eqref{SubExponential} with $0<\alpha<1$. Let either $\sfL=\sfL_1$ and $\beta=2(1-\alpha)$ or
$\sfL=\sfL_2$ assuming~\eqref{hyp:b_mass},~\eqref{hyp:b_bounds}. Then there exists $\C>0$ such that
\[
     -\bangle{\sfL f,f} \ge \C\,\|f-\rho_f F\|_{-\beta}^2 \qquad\qquad \forall\,f\in\mathcal D(\R^{2d}) \,.
\]
\end{lemma}
\begin{proof}
For $\sfL=\sfL_1$ the result is a consequence of the \emph{weighted Poincar\'e inequality}
\be{weightedPoincare}
  \intw{|\nabla_v g|^2 \M}\mu \ge \C\intw{(g-\overline g)^2\, \bangle v^{-2\,(1-\gam)}\M}v \qquad
  \forall\,g\in\mathcal D(\R^d)\,,
\ee
which will be proved in Appendix~\ref{Appendix:A}.

For $\sfL=\sfL_2$ we estimate
\begin{align*}
  &\rint (f-\rho_f F)^2 \bangle{v}^{-\beta} \dd\mu = \rint \left( \rint \big(fF'-f'F\big)\,F'\dd\mu' \right)^2 \bangle{v}^{-\beta}\dd\mu \\
  & \le \rint F\bangle{v}^\beta \dd v \iintw{\big(fF'-f'F\big)^2 \bangle{v}^{-\beta} \bangle{v'}^{-\beta}}{\mu \dd\mu'} 
  \le - \frac{1}{\C} \rint (\sfL_2 f)f \,\dd\mu
\end{align*}
with 
\[
  \C = \frac{\underline \bb}{2} \left( \rint F\bangle{v}^\beta \dd v\right)^{-1}\,.
\]
For the first inequality we have used Cauchy-Schwarz and for the second,~\eqref{L2-diss} and the hypothesis 
\eqref{hyp:b_bounds}. Integration with respect to $x$ completes the proof.
\end{proof}

Apart from proving the weighted Poincar\'e inequality~\eqref{weightedPoincare} (in Appendix~\ref{Appendix:A}), we shall 
also show (in Appendix~\ref{Appendix:B}) how it can be used to prove algebraic decay to equilibrium for the spatially
homogeneous equation with $\sfL=\sfL_1$ and $0<\alpha<1$, \emph{i.e.} the Fokker-Planck equation with sub-exponential
equilibrium. The loss of information due to the weight $\bangle{v}^{-2(1-\alpha)}$ has to be compensated
by a $\mathrm L^2$-bound for the initial datum with a weight $\bangle{v}^k$, $k>0$, as in Theorem~\ref{th:main}.
For this problem, estimates based on \emph{weak Poincar\'e inequalities} are also very popular in the scientific community of semi-group theory and Markov processes (see~\cite{MR1856277},~\cite[Proposition~7.5.10]{MR3155209},~\cite{hal-01241680} and Appendix~\ref{Appendix:B}). Estimates based on weak Poincar\'e inequalities rely on a uniform bound for the initial data for $\gam<1$ which is not needed for $\gam\ge1$, while the approach developed in this paper 
provides a continuous transition from the range $0<\alpha<1$ to the range $\alpha\ge1$ since we may choose $k\searrow 0$ as $\gam\nearrow 1$. Note that for $\gam=1$, the weighted Poincar\'e inequality~\eqref{weightedPoincare} reduces to the Poincar\'e inequality~\eqref{Poincare}.

The proof of Theorem~\ref{th:main} follows along the lines of the hypocoercivity approach of~\cite{MR2576899,MR3324910} and its extension to cases without confinement as in~\cite{hal-01575501,bouin:hal-01991665}. It combines information on the 
\emph{microscopic} and the \emph{macroscopic} dissipation properties. The essence of the microscopic information is
given in Lemma~\ref{lem:micro-coerc}. Since the macroscopic limit of~\eqref{eq:main} is the heat equation on the whole
space, it is natural that for the estimation of the macroscopic dissipation we use \emph{Nash's inequality,} a tool 
which has been developed for this purpose. The result of Theorem~\ref{th:main} can be interpreted as giving the weaker of the microscopic decay rate $t^{-k/\beta}$ and the macroscopic decay rate $t^{-d/2}$. Only for $k\ge \beta\,d/2$ the decay rate of the macroscopic diffusion limit is recovered.

Related results have been shown in~\cite{hal-01575501,bouin:hal-01991665,hal-01697058}, where the latter is somewhat 
complementary to this work dealing with Gaussian local equilibria in the presence of an external potential with 
sub-exponential growth in the variable~$x$. Also see~\cite{Wang_2015,MR3744403,MR3892316} for various earlier results dealing with external potentials with a growth like $\bangle x^\gamma$, $\gamma<1$, based on weak Poincar\'e inequalities, spectral techniques, $\mathrm H^1$ hypocoercivity methods, \emph{etc}.

This paper is organized as follows. In Section~\ref{Sec:Entropy}, we prove an hypocoercive estimate relating a modified
entropy, which is equivalent to $\|f\|^2$, to an entropy production term involving a microscopic and a macroscopic component. Using weighted $\mathrm L^2$-estimates established in Section~\ref{Sec:wL2}, we obtain a new control by the microscopic component in Lemma~\ref{lem:Psi} while the macroscopic component is estimated as in~\cite{hal-01575501} using Nash's inequality, see Lemma~\ref{lem:Phi}. By collecting these estimates in Section~\ref{Sec:Proof}, we complete the proof of Theorem~\ref{th:main}. Two appendices are devoted to $\sfL=\sfL_1$: in Appendix~\ref{Appendix:A} we provide a new proof of~\eqref{weightedPoincare} and comment on the interplay with weak Poincar\'e inequalities, while the spatially homogeneous version of~\eqref{eq:main} is dealt with in Appendix~\ref{Appendix:B} and rates of relaxation towards the local equilibrium are discussed using weighted $\mathrm L^2$-norms, as an alternative approach to the weak Poincar\'e inequality method of~\cite{hal-01241680}. The main novelty of our approach is that we use new interpolations in order to exploit the entropy production term. As a consequence, with the appropriate weights, no other norm is needed than weighted $\mathrm L^2$-norms. For simplicity, we assume that the distribution function is nonnegative but the extension to sign changing functions is straightforward.

\section{An entropy--entropy production estimate}\label{Sec:Entropy}

We adapt the strategy of~\cite{MR3324910,hal-01575501}, denoting by $\op{T} = v\cdot\nabla_x$ the free streaming 
operator and by $\sfPi$ the orthogonal projection on $\mathrm{Ker}(\op L)$ in $\mathrm L^2(\R^d,\dmu)$, given by
\[
\sfPi f:=\rho_f\,\M\qquad\mbox{where }\rho_f=\rint f\dd v\,.
\]
To build a suitable Lyapunov functional, we introduce the operator
\[
\op A :=\big(\Id+(\op{T\Pi})^*(\op{T\Pi})\big)^{-1}(\op{T\Pi})^*
\]
and consider
\[
\sfH[f] :=\frac12\,\|f\|^2+\delta\,\scalar{\sfA f}f\,.
\]
It is known from~\cite[Lemma 1]{MR3324910} that, for any $\cst\in (0,1)$, $\sfH[f]$ and $\|f\|^2$ are equivalent in the sense that
\be{H-norm}
\frac12\,(1-\delta)\,\|f\|^2\le\sfH[f]\le\frac12\,(1+\delta)\,\|f\|^2\,.
\ee
A direct computation shows that
\be{E-EP}
\dt\sfH[f]=-\,\op D[f]
\ee
with
\begin{align*}\label{def:D}
\op D[f]:=&-\,\scalar{\sfL f}f+\cst\,\scalar{\op{AT\Pi}f}{\op\Pi f}\\
&+\cst\,\scalar{\op A\op T(\Id-{\op\Pi})f}{\op\Pi f}-\,\cst\,\left\langle\op{TA}(\Id-\op\Pi)f,(\Id-\op\Pi)f\right\rangle-\cst\,\scalar{\op{AL}(\Id-\op\Pi)f}{\op\Pi f}
\end{align*}
where we have used that $\left\langle\op Af,\op Lf\right\rangle=0$. Note that in terms of the new notation the result of 
Lemma~\ref{lem:micro-coerc} reads
\be{eq:term_1}
\bangle{\sfL f,f}\le-\,\C\,\nrm{(\Id-\sfPi)f}{-\beta}^2\,.
\ee
\begin{proposition}\label{prop:entropy_ineq} Under the assumptions of Theorem~\ref{th:main} and for small enough
$\delta>0$, there exists $\kappa>0$ such that, for any $f\in\mathrm L^2\(\bangle v^{-\beta}\dd x\ddmu\)\cap\mathrm L^1(\d x\dd v)$,
\[
\op D[f]\ge\kappa\(\nrm{(\Id-\sfPi)f}{-\beta}^2+\scalar{\op{AT\Pi}f}{\op\Pi f}\)\,.
\]
\end{proposition}
Note that $\kappa$ does not depend on $k>0$ (the parameter $k$ appears in the assumptions of Theorem~\ref{th:main}). 
An estimate of $\op D[f]$ in terms of $\scalar{\op{AT\Pi}f}{\op\Pi f}$ and $\nrm{(\Id-\sfPi)f}{}^2$ has already been derived in~\cite{hal-01575501}, but using the weighted norm $\nrm{(\Id-\sfPi)f}{-\beta}$ is a new idea.
\begin{proof} We have to prove that the three last terms in $\op D[f]$ are controlled by the first two. The main difference with~\cite{MR3324910,hal-01575501} is the additional weight $\bangle v^{-\beta}$ in the velocity variable.

\step{Step 1: rewriting $\scalar{\op{AT\Pi}f}{\op\Pi f}$.}
Let $u=u_f$ be such that
\[
u\,\M=\(\Id+(\op{T\Pi})^*(\op{T\Pi})\)^{-1}\op\Pi f\,.
\]
Then $u$ solves $(u-\Theta\,\Delta u)\,\M=\sfPi f$, that is,
\be{eq:u}
u-\Theta\,\Delta u=\rho_f \,,
\ee
where $\Theta:=\rint|v\cdot\mathsf e|^2\,\M(v)\dd v$ for an arbitrary unit vector $\mathsf e$. Since
\begin{align*}
\op{AT\Pi}f &=\(\Id+(\op{T\Pi})^*(\op{T\Pi})\)^{-1}(\op{T\Pi})^*(\op{T\Pi})\,\op\Pi f\\
&=\big(\Id+(\op{T\Pi})^*(\op{T\Pi})\big)^{-1}\big(\Id+(\op{T\Pi})^*(\op{T\Pi})-\Id\big)\,\op\Pi f\\
&=\op\Pi f-\big(\Id+(\op{T\Pi})^*(\op{T\Pi})\big)^{-1}\,\op\Pi f=\sfPi f-u\,\M=(\rho_f-u)\,\M\,,
\end{align*}
then by using equation~\eqref{eq:u}, we obtain
\[
\scalar{\op{AT\Pi}f}{\op\Pi f}=\scalar{\op\Pi f-u\,\M}{\op\Pi f}=\scalar{-\Theta\,\Delta u\,\M}{(u-\Theta\,\Delta u)\,\M}\,,
\]
from which we deduce
\be{eq:ATPi}
\scalar{\op{AT\Pi}f}{\op{\Pi}f}=\Theta\nrm{\nabla u}{\mathrm L^2(\dd x)}^2+\Theta^2\nrm{\Delta u}{\mathrm L^2(\dd x)}^2 \ge 0\,.
\ee

\step{Step 2: a bound on $\scalar{\op A\op T(\Id-{\op\Pi})f}{\op\Pi f}$.}
If $u$ solves~\eqref{eq:u}, we use the fact that
\be{eq:A*}
\sfA^*\sfPi f=\sfT\sfPi\,u\,\M=\sfT u\,\M
\ee
to compute
\[
\scalar{\op A\op T(\Id-\op\Pi)f}{\op\Pi f}=\scalar{(\Id-\op\Pi)f}{\sfT^*\op A^*\op\Pi f}=\scalar{(\Id-\op\Pi)f}{\sfT^*\sfT u\,\M}\,.
\]
Therefore, since $\op T^*\op{T}u\,\M=-\,v \cdot \Dx\(v \cdot\Dx u\)\M$, the Cauchy-Schwarz inequality yields
\begin{multline*}
\left| \scalar{\op A\op T(\Id-\op\Pi)f}{\op\Pi f}\right| \le\nrm{(\Id-\sfPi)f}{-\beta}\,\nrm{\sqrt\M\,\bangle v^{\frac\beta2} v \cdot \Dx\(v \cdot\Dx u\)\,\sqrt\M\,}{\mathrm L^2(\d x\dd v)}\\
\le\Theta_{\beta+4}\,\nrm{(\Id-\sfPi)f}{-\beta}\,\nrm{\Delta u}{\mathrm L^2(\d x)}\,,
\end{multline*}
hence
\be{eq:term_3}
\left| \scalar{\op A\op T(\Id-\op\Pi)f}{\op\Pi f}\right|\le\C_4\,\nrm{(\Id-\sfPi)f}{-\beta}\,\scalar{\op{AT\Pi}f}{\op{\Pi}f}^\frac12
\ee
where we have used identity~\eqref{eq:ATPi}, $\C_4=\Theta_{\beta+4}/\Theta$ and
\[
\Theta_k:=\rint\bangle v^k\,\M(v)\dd v\,.
\]
With this convention, note that $\Theta_2=1+d\,\Theta$.

\step{Step 3: estimating $\left\langle\op{TA}(\Id-\op\Pi)f,(\Id-\op\Pi)f\right\rangle$.}
As noted in~\cite[Lemma 1]{MR3324910}, the equation $g=\sfPi g=\sfA f$ is equivalent to
\[
\big(\Id+(\op{T\Pi})^*(\op{T\Pi})\big)g=(\op{T\Pi})^*f
\]
which, after multiplying by $g$ and integrating, yields
\begin{align*}
\n g^2+\n{\op Tg}^2&=\bangle{g, g+(\op{T\Pi})^*(\op{T\Pi}) g}\\
&=\bangle{g,(\op{T\Pi})^*f}=\bangle{\sfT\sfPi g,f}=\bangle{\sfT\sfA f,f}\le\nrm{(\Id-\op\Pi)f}{-\beta}\,\nrm{\sfT\sfA f}\beta
\end{align*}
by the Cauchy-Schwarz inequality. We know that $(\op{T\Pi})^*=-\,\sfPi\sfT$ so that $\sfA f=g=w\,\M$ is determined by the equation
\[
w-\Theta\,\Delta w=-\,\Dx\cdot\rint v\,f\dd v\,.
\]
After multiplying by $w$ and integrating in $x$, we obtain that
\begin{multline*}
\Theta\rint|\Dx w|^2\dd x\le\rint|w|^2\dd x+\Theta\rint|\Dx w|^2\dd x\\
\le\(\rint|\Dx w|^2\dd x\)^\frac12\(\rint\left|{\textstyle\rint v\,f\dd v}\right|^2\dd x\)^\frac12
\end{multline*}
and note that
\begin{multline*}
\rint\left|{\textstyle\rint v\,f\dd v}\right|^2\dd x=\rint\left|\rint\bangle v^{-\frac\beta2}\,\frac{(\Id-\op\Pi)f}{\sqrt\M}\cdot|v|\,\bangle v^\frac\beta2\,\sqrt\M\dd v\right|^2\dd x\\
\le\Theta_{\beta+2}\,\nrm{(\Id-\op\Pi)f}{-\beta}^2
\end{multline*}
by the Cauchy-Schwarz inequality. Hence
\[
\rint|\Dx w|^2\dd x\le\frac{\Theta_{\beta+2}}{\Theta^2}\,\nrm{(\Id-\op\Pi)f}{-\beta}^2
\]
and
\[
\nrm{\sfT\sfA f}\beta^2=\n{\Dx w\cdot(v\,\bangle v^{\beta/2}\,\M)}^2=\Theta_{\beta+2}\rint|\Dx w|^2\dd x\le\C_2^2\,\nrm{(\Id-\op\Pi)f}{-\beta}^2
\]
with $\C_2:=\Theta_{\beta+2}/\Theta$. Since $g=\op A f$ so that $\n{\sfA f}^2+\n{\sfT\sfA f}^2=\n g^2+\n{\sfT g}^2$, we obtain that
\be{eq:term_4}
\bangle{\sfT\sfA f,f}=\left\langle\op{TA}(\Id-\op\Pi)f,(\Id-\op\Pi)f\right\rangle\le\nrm{(\Id-\op\Pi)f}{-\beta}\,\nrm{\sfT\sfA f}\beta\le\C_2\,\nrm{(\Id-\sfPi)f}{-\beta}^2\,.
\ee
We also remark that
\begin{multline*}
\bangle{\sfT\sfA f,f}=\scalar{(v\cdot\Dx w)\,\M}f=\rint\Dx w\cdot\(\rint v\,f\dd v\)\dd x\\
=\rint|w|^2\dd x+\Theta\rint|\Dx w|^2\dd x\ge0\,.
\end{multline*}

\step{Step 4: bound for $\scalar{\op{AL}(\Id-\op\Pi)f}{\op\Pi f}$.}
We use again identity~\eqref{eq:A*} to compute
\begin{multline*}
\left|\scalar{\op{AL}(\Id-\op\Pi)f}{\op\Pi f}\right|=\left|\bangle{(\Id-\op\Pi)f,\op{L^*A^*}\sfPi f}\right|=\left|\bangle{(\Id-\op\Pi)f,\op{L^*T}u\,\M}\right|\\\le\nrm{(\Id-\sfPi)f}{-\beta}\,\nrm{\op{L^*T}u\,\M}\beta\,.
\end{multline*}
In case $\sfL=\sfL_1$ we remark that
\begin{multline*}
\nrm{\op{L_1^*T}u\,\M}\beta^2=\riint\left|\Dv\cdot\big(\M\,\Dv\(v\cdot\Dx u\)\big)\right|^2\,\bangle v^\beta\d x\ddmu\\
=\riint|\Dv\M\cdot\Dx u|^2\,\bangle v^\beta\dd x\ddmu\le\nrm{\Dv\M}{\mathrm L^2\(\bangle v^\beta\dmu\)}^2\,\nrm{\Dx u}{\mathrm L^2(\dd x)}^2\,.
\end{multline*}
In case $\sfL=\sfL_2$, note first that
\[
(\op{L}_2^*\op{T}u\,\M)(v)=\(\rint\bb(v',v)\,(v'-v)\,\M(v')\dd v'\)\cdot\Dx u\,\M(v) \,,
\]
and thus, by the Cauchy-Schwarz inequality,
\[
\nrm{\op{L_2^*T}u\,\M}\beta\le\mathcal B\,\nrm{\Dx u}{\mathrm L^2(\dd x)}\,,\qquad\mbox{with}\quad\mathcal B=\nrm{\rint\bb(v',v)\,(v'-v)\,\M'\,\M\dd v'}{\mathrm L^2\(\bangle v^\beta\dmu\)} \,.
\]
For proving finiteness of $\mathcal B$ we use~\eqref{hyp:b_bounds} in
\begin{multline*}
   \left| \rint\bb(v',v)\,(v'-v)\,\M'\dd v' \right|\\
   \le \overline\bb \int_{|v'-v|<1} |v'-v|^{1-\gamma} \M'\dd v'
   + \overline\bb \int_{|v'-v|>1} |v'-v|^{1-\beta} \M'\dd v'
   \le c\left(1 + \bangle{v}^{1-\beta}\right) \,,
\end{multline*}
which implies
\[
   \mathcal{B}^2 \le c^2 \rint \left(1 + \bangle{v}^{1-\beta}\right)^2 \bangle{v}^\beta F \dd v <\infty \,.
\]
Combining these estimates with identity~\eqref{eq:ATPi} we get
\be{eq:term_5}
\left|\scalar{\op{AL}(\Id-\op\Pi)f}{\op\Pi f}\right|\le\C_\M\,\nrm{(\Id-\sfPi)f}{-\beta}\,\scalar{\op{AT\Pi}f}{\op{\Pi}f}^\frac12
\ee
where $\C_\M=\mathcal B/\sqrt\Theta$.

\step{Step 5: collecting all estimates.}
Altogether, combining~\eqref{eq:term_1} and~\eqref{eq:term_3},~\eqref{eq:term_4} and~\eqref{eq:term_5}, we obtain
\begin{multline*}
\dt\sfH[f]\le-\,\C\,\nrm{(\Id-\sfPi)f}{-\beta}^2-\cst\,\scalar{\op{AT\Pi}f}{\op\Pi f}\\
+\cst\(\C_4+\C_\M\)\nrm{(\Id-\sfPi)f}{-\beta}\,\scalar{\op{AT\Pi}f}{\op\Pi f}^\frac12+\cst\,\C_2\,\nrm{(\Id-\sfPi)f}{-\beta}^2
\end{multline*}
which by Young's inequality yields the existence of $\kappa>0$ such that
\[
\dt\sfH[f]\le-\,\kappa\(\nrm{(\Id-\sfPi)f}{-\beta}^2+\scalar{\op{AT\Pi}f}{\op\Pi f}\)
\]
for some $\delta\in(0,1)$. Indeed, with $X:=\nrm{(\Id-\sfPi)f}{-\beta}$ and $Y:=\scalar{\op{AT\Pi}f}{\op\Pi f}^\frac12$, it is enough to check that the quadratic form
\[
\mathcal Q(X,Y):=(\C-\delta\,\C_2)\,X^2-(\C_4+\C_\M)\,X\,Y+\delta\,Y^2
\]
is positive, \emph{i.e.}, $\mathcal Q(X,Y)\ge\kappa\,(X^2+Y^2)$ for some $\kappa=\kappa(\delta)$ and $\delta\in(0,1)$. \end{proof}

\section{Weighted \texorpdfstring{$\mathrm L^2$}{L2} estimates}\label{Sec:wL2}

In this section, we show the propagation of weighted norms with weights $\bangle v^k$ of arbitrary positive order $k\in\R^+$.
\begin{proposition}\label{prop:propag_L2_m} Let $k>0$ and $f$ be solution of~\eqref{eq:main} with $f^\init\in\mathrm L^2(\bangle v^k\dd x\ddmu)$. Then there exists a constant $\mathcal K_k>1$ such that
\[
\forall\,t\ge0\,\quad\nrm{f(t,\cdot,\cdot)}k\le\mathcal K_k\,\nrm{f^\init}k\,.
\]
\end{proposition}
We recall that $\nrm fk$ is defined by~\eqref{Norms}. We shall state a technical lemma (Lemma~\ref{lem:Lyapunov} below) before proving a splitting result in Lemma~\ref{lem:splitting}, from which the proof of Proposition~\ref{prop:propag_L2_m} easily follows (see section~\ref{Sec:PrfProp3}).

\subsection{A technical lemma}
\begin{lemma}\label{lem:Lyapunov} If either $\sfL=\sfL_1$ or $\sfL=\sfL_2$, then
there exists $\ell>0$ for which, for any $k\ge0$, there exist $a_k,b_k,R_k>0$ such that
\be{eq:Lyapunov}
\scalar{\sfL f}{f\bangle v^k} \le \riint\(a_k\,\mathds1_{|v|<R_k} - b_k\,\bangle v^{-\ell}\)\,|f|^2\,\bangle v^k\dd x\ddmu \,,
\ee
for any $f\in\mathcal{D}(\R^{2d})$.
\end{lemma}

\begin{proof} In the Fokker-Planck case $\sfL=\sfL_1$ we have
\begin{eqnarray*}
   \rint \sfL_1 f\,f\bangle v^k \dd\mu &=& - \rint \left| \nabla_v \(\frac{f}{F}\)\right|^2 \bangle{v}^k F \dd v
   - \frac{1}{2} \rint \nabla_v \(\frac{f^2}{F^2}\) \cdot \nabla_v \bangle{v}^k F \dd\mu \\
   &&\hspace*{-2cm}\le\,\frac{1}{2} \rint f^2 \left( \frac{\nabla_v F}{F}\cdot\nabla_v\bangle{v}^k + \Delta_v \bangle{v}^k\right) \dd\mu \\
   &&\hspace*{-1cm}=\,\frac{k}{2} \rint f^2 \bangle{v}^{k-4}\left( 2-k + (d+k-2)\bangle{v}^2 + \alpha \bangle{v}^\alpha 
     - \alpha\bangle{v}^{\alpha+2}\right) \dd\mu \\
    &\le& \frac{k}{2} \rint f^2 \bangle{v}^{k-2}\left( c_k - \alpha\bangle{v}^\alpha\right) \dd\mu \\
    && =\rint\(a_k\,\mathds1_{|v|<R_k} - b_k\,\bangle v^{-\ell}\)\,|f|^2\,\bangle v^k \dd\mu \\
   &&\hspace*{0.5cm}+\,\frac{k}{2} \rint f^2 \bangle{v}^{k-2} \left( c_k\left( 1 - \mathds1_{|v|<R_k} \bangle{v}^2\right) 
    - \frac{\alpha}{2}\bangle{v}^\alpha\right) \dd\mu  \,,
\end{eqnarray*}
with $c_k =  |k-2| + |d+k-2| + \alpha$, $a_k = c_k\,k/2$, $b_k = \alpha\,k/4$, $\ell=2-\alpha$. The choice 
$R_k = \left( 2\,c_k/\alpha\right)^{1/\alpha}$ makes the last term negative, which completes the proof.

In the case of the scattering operator $\sfL=\sfL_2$, with $h:=f/\M$, we have
\begin{align*}
2\rint f\,\sfL_2 f\,\bangle v^k\ddmu&=2\riint\bb(v,v')\(h'-h\)h\,\bangle v^k\,\M\,\M'\dd v\dd v'\\
&=\riint\bb(v,v')\,\left(2\,h'h - h^2\right)\,\bangle v^k\,\M\,\M'\dd v\dd v'\\
&\hspace*{2cm}-\riint\bb(v',v)\, h^2 \bangle v^k\,\M\,\M'\dd v\dd v'\,,
\end{align*}
where we have used~\eqref{hyp:b_mass}. Swapping $v$ and $v'$ in the last integral gives
\begin{align*}
2\rint f\,\sfL_2 f\,\bangle v^k\ddmu=& -\riint\bb(v,v')\(h-h'\)^2\bangle v^k\,\M\,\M'\dd v\dd v'\\
&+\riint\bb(v,v')\, (h')^2 \left( \bangle v^k - \bangle{v'}^k\right)\M\,\M'\dd v\dd v'\\
\le&\rint \left(\rint\bb(v',v) \left(\bangle{v'}^k - \bangle v^k\right)\M'\dd v' \right) f^2 \dd\mu\,,
\end{align*}
with another swap $v\leftrightarrow v'$ in the last step. Now we employ~\eqref{hyp:b_bounds} and its consequence~\eqref{nu2-est}:
\begin{eqnarray*}
  \rint\bb(v',v) \left(\bangle{v'}^k - \bangle v^k\right)\M'\dd v'  &=&  \rint\bb(v',v) \bangle{v'}^k \M'\dd v'  - \bangle{v}^k \nu_2(v) \\
  &\le& 2\,a_k \bangle{v}^{-\beta} - \underline\nu \bangle{v}^{k-\beta} \,,
\end{eqnarray*}
where the estimation of the first term is analogous to the derivation of~\eqref{nu2-est}. This implies
\begin{eqnarray*}
   \rint \sfL_2 f\,f\bangle v^k \dd\mu &\le& \rint\(a_k\,\mathds1_{|v|<R_k} - b_k\,\bangle v^{-\ell}\)\,|f|^2\,\bangle v^k \dd\mu \\
   && + \rint f^2 \bangle{v}^k \left( a_k\left( \bangle{v}^{-\beta-k} - \mathds1_{|v|<R_k}\right) + b_k \bangle{v}^{-\ell}
    - \frac{\underline\nu}{2} \bangle{v}^{-\beta} \right) \dd v  \,.
\end{eqnarray*}
The last term is made negative by the choices $\ell=\beta$, $b_k = \underline\nu\,/4$, $R_k = (4\,a_k/\underline\nu)^{1/k}$.
\end{proof}

\subsection{A splitting result}

As in~\cite{MR3779780,hal-01241680,MR3488535}, we write $\sfL-\sfT$ as a dissipative part $\sfC$ and a bounded part $\sfB$ such that $\sfL-\sfT=\sfB+\sfC$.
\begin{lemma}\label{lem:splitting} With the notation of Lemma~\ref{lem:Lyapunov}, let $k_1>0$, $k_2>k_1+2\,\ell$, $a=\max\{a_{k_1},a_{k_2}\}$, $R=\max\{R_{k_1},R_{k_2}\}$, $\sfC=a\,\mathds1_{|v|<R}$ and $\sfB=\sfL-\sfT-\sfC$. For any 
$t\ge 0$ we have:
\begin{enumerate}
\item[(i)] $\|\sfC\|_{\mathrm L^2(\d x\ddmu)\to\mathrm L^2\(\bangle v^{k_2}\dd x\ddmu\)}\le a\,\bangle R^{k_2/2}$,
\item[(ii)] $\|e^{t\sfB}\|_{\mathrm L^2\(\bangle v^{k_1}\dd x\ddmu\)\to\mathrm L^2\(\bangle v^{k_1}\dd x\ddmu\)}\le1$,
\item[(iii)] $\|e^{t\sfB}\|_{\mathrm L^2\(\bangle v^{k_2}\dd x\ddmu\)\to\mathrm L^2\(\bangle v^{k_1}\dd x\ddmu\)}\le C\(1+t\)^{-\frac{k_2-k_1}{2\,\ell}}$ for some $C>0$.
\end{enumerate}
\end{lemma}
\begin{proof} Property (i) is a consequence of the definition of $\sfC$. Property (ii) follows from Lemma~\ref{lem:Lyapunov} according to
\begin{align*}
\riint f\,\sfB f\,\bangle v^{k_1}\dd x\ddmu &\le\riint\(a_{k_1}\,\mathds1_{B_{R_{k_1}}}-a\,\mathds1_{B_R}-b_{k_1}\bangle v^{-\ell}\)|f|^2\,\bangle v^{k_1}\dd x\ddmu\\
&\le-\,b_{k_1} \riint|f|^2\,\bangle v^{{k_1}-\ell}\dd x\ddmu\,.
\end{align*}
Similarly, we know that $\|e^{t\sfB}\|_{\mathrm L^2\(\bangle v^{k_2}\dd x\ddmu\)\to\mathrm L^2\(\bangle v^{k_2}\dd x\ddmu\)}\le1$.

By combining H\"older's inequality
\[
\nrm f{k_1}^2\le\nrm f{k_1-\ell}^{\frac{2\,(k_2-k_1)}{k_2-k_1+\ell}}\,\nrm f{k_2}^\frac{2\,\ell}{k_2-k_1+\ell}
\]
with~Property (ii), we obtain
\begin{align*}
\riint f\,\sfB f\,\bangle v^{k_1}\dd x\ddmu&\le-\,b_{k_1}\,\nrm f{k_1}^{2\,\big(1+\frac{\ell}{k_2-k_1}\big)}\,\nrm{f^\init}{k_2}^{-\,\frac{2\,\ell}{k_2-k_1}}\,.
\end{align*}
With $f=e^{t\sfB}\,f^\init$, Property (iii) follows from Gr\"onwall's lemma according to
\[
\nrm f{k_1}^2\le\(\nrm{f^\init}{k_1}^{-\,\frac{2\,\ell}{k_2-k_1}}+\tfrac{2\,\ell\,b_{k_1}\,t}{k_2-k_1}\nrm{f^\init}{k_2}^{-\,\frac{2\,\ell}{k_2-k_1}}\)^{-\,\frac{k_2-k_1}{\ell}}\le\(\tfrac{k_2-k_1}{k_2-k_1+2\,\ell\,b_{k_1}\,t}\)^{\frac{k_2-k_1}{\ell}}\,\nrm{f^\init}{k_2}^2\,.\]
\end{proof}

\subsection{Proof of Proposition\texorpdfstring{~\ref{prop:propag_L2_m}}{3}}\label{Sec:PrfProp3}

Using the convolution $\op{U}\star \op{V}=\int_0^t \op{U}(t-s)\,\op{V}(s)\dd s$, Duhamel's formula asserts that
\[
e^{t(\sfL-\sfT)}=e^{t\sfB}+e^{t\sfB}\star\,\sfC\,e^{t(\sfL-\sfT)}\,.
\]
Therefore, by 
\[
   \|e^{t(\sfL-\sfT)}\|_{\mathrm L^2\(\dd x\ddmu\)\to\mathrm L^2\(\dd x\ddmu\)}\le1
   \qquad\Rightarrow\qquad
   \|e^{t(\sfL-\sfT)}\|_{\mathrm L^2\(\bangle v^{k_1}\dd x\ddmu\)\to\mathrm L^2\(\dd x\ddmu\)}\le 1 \,,
 \]
by Lemma~\ref{lem:splitting}, and with $k=k_1$, $\ell$ as in Lemma~\ref{lem:Lyapunov} and $k_2>k+2\,\ell$, we get that
\[
\n{e^{t(\sfL-\sfT)}}_{\mathrm L^2\(\bangle v^{k_1}\dd x\ddmu\)\to\mathrm L^2\(\bangle v^{k_1}\dd x\ddmu\)}\le 1+a\,\bangle R^\frac{k_2}2\int_0^t\frac{C\,\d s}{\(1+s\)^{\frac{k_2-k_1}{2\,\ell}}}
\]
is bounded uniformly in time.\qed

\section{Proof of Theorem~\texorpdfstring{\ref{th:main}}1}\label{Sec:Proof}

The control of the macroscopic part $\sfPi f$ by $\scalar{\op{AT\Pi}f}{\op\Pi f}$ is achieved as in~\cite{hal-01575501}. We sketch a proof for the sake of completeness.
\begin{lemma}\label{lem:Phi} Under the assumptions of Theorem~\ref{th:main}, for any $f\in\mathrm L^1(\d x\ddmu)\cap\mathrm L^2(\d x\dd v)$,
\[
\scalar{\op{AT\Pi}f}{\op\Pi f}\ge\Phi\(\|\op{\Pi}f\|^2\)
\]
with
\[
\Phi^{-1}(y) :=2\,y+\(\frac y{\sfc}\)^\frac d{d+2}\,,\quad\sfc=\Theta\,\C_{\rm{Nash}}^{-\frac{d+2}d}\,\|f\|_{\mathrm L^1(\d x\dd v)}^{-\frac4d}\,.
\]
\end{lemma}
\begin{proof} With $u$ defined by~\eqref{eq:u}, we control~$\n{\op{\Pi}f}^2=\nrm{\rho_f}{\mathrm L^2(\d x)}^2$ by $\scalar{\op{AT\Pi}f}{\op\Pi f} $ according~to
\[
\n{\op{\Pi}f}^2=\nrm u{\mathrm L^2(\dd x)}^2+2\,\Theta\nrm{\nabla u}{\mathrm L^2(\dd x)}^2+\Theta^2\nrm{\Delta u}{\mathrm L^2(\dd x)}^2
\le\,\nrm u{\mathrm L^2(\dd x)}^2+2\,\langle\op{AT\Pi}f,f\rangle
\]
using~\eqref{eq:ATPi}. Then we observe that
\[
\nrm{u}{\mathrm L^1}=\nrm{\rho_f}{\mathrm L^1}=\|f\|_{\mathrm L^1(\d x\dd v)}\,,\quad\nrm u{\mathrm L^2(\d x)}^2\le\frac1\Theta\,\bangle{\op{AT\Pi}f,f}
\]
and use Nash's inequality
\[
\nrm u{\mathrm L^2(\dd x)}^2\le\C_{\rm{Nash}}\,\nrm u{\mathrm L^1(\dd x)}^\frac4{d+2}\,\nrm{\nabla u}{\mathrm L^2(\dd x)}^\frac{2\,d}{d+2}
\]
to conclude the proof.\end{proof}

The control of $(\Id-\sfPi)f$ by the entropy production term relies on a simple new estimate.
\begin{lemma}\label{lem:Psi} Under the assumptions of Theorem~\ref{th:main}, for any solution $f$ of~\eqref{eq:main} with initial datum  $f^\init\in\mathrm L^2(\bangle v^k\dd x\ddmu)\cap\mathrm L^1(\d x\dd v)$, we have
\[
\nrm{(\Id-\sfPi)f(t,\cdot,\cdot)}{-\beta}^2\ge\Psi\(\n{(\Id-\sfPi)f(t,\cdot,\cdot)}^2\)
\]
for any $t\ge0$, where $\mathcal K_k$ is as in Proposition~\ref{prop:propag_L2_m} and
\[
\Psi(y):=C_0\,y^{1+\beta/k}\,,\quad C_0:=\Big(\mathcal K_k\,\big(1+\Theta_k\big)\,\|{f^\init}\|_k\Big)^{-\frac{2\,\beta}k}\,.
\]\end{lemma}
\begin{proof} H\"older's inequality
\[
\n{(\Id-\sfPi)f}\le\nrm{(\Id-\sfPi)f}{-\beta}^\frac k{k+\beta}\,\nrm{(\Id-\sfPi)f}k^\frac\beta{k+\beta}
\]
and
\[
\nrm{(\Id-\sfPi)f}k\le\|f\|_k+\Theta_k\nrm{\rho}{\mathrm L^2(\dd x)}\le(1+\Theta_k)\,\|f\|_k\le\mathcal K_k\,(1+\Theta_k)\,\nrm{f^\init}k\,,
\]
where the last inequality holds by Proposition~\ref{prop:propag_L2_m}, provide us with the estimate.\end{proof}

\noindent\emph{Proof of Theorem~\ref{th:main}.} Using the estimates of Lemma~\ref{lem:Phi} and Lemma~\ref{lem:Psi}, we obtain that
\[
\nrm{(\Id-\sfPi)f}{-\beta}^2+\scalar{\op{AT\Pi}f}{\op\Pi f}\ge\Psi\(\n{(\Id-\sfPi)f}^2\)+\Phi\(\n{\op{\Pi}f}^2\)\,.
\]
Using~\eqref{E-EP} and the fact that $\op D[f]\ge 0$ by Proposition~\ref{prop:entropy_ineq}, we know that
\[
\n{(\Id-\sfPi)f}^2\le z_0\quad\mbox{and}\quad\|\op{\Pi}f\|^2\le z_0\quad\mbox{where}\quad z_0:=\|f^\init\|^2\,.
\]
Thus, from
\[
\Phi^{-1}(y)=2\,y+\(\frac y{\sfc}\)^\frac d{d+2}\le\(C_1^{-1}\,y\)^{\frac d{d+2}}\quad\mbox{with}\quad C_1:=\(2\,\Phi(z_0)^{\frac 2{d+2}}+\sfc^{-\frac d{d+2}}\)^{-\frac{d+2}d}\,,
\]
as long as $y\le\Phi(z_0)$, we obtain
\[
\Phi\(\|\op{\Pi}f\|^2\)\ge C_1\,\|\op{\Pi}f\|^{\,2\,\frac{d+2}d}\,,
\]
since $\|\op{\Pi}f\|^2\le z_0$. As a consequence,
\begin{multline*}
\nrm{(\Id-\sfPi)f}{-\beta}^2+\scalar{\op{AT\Pi}f}{\op\Pi f}\ge C_0\,\n{(\Id-\sfPi)f}^{\,2\,\frac{k+\beta}k}+C_1\,\|\op{\Pi}f\|^{\,2\,\frac{d+2}d}\\
\ge\min\big\{C_0\,z_0^{\frac\beta k-\frac1\zeta},\,C_1\,z_0^{\frac2d-\frac1\zeta}\big\}\,\|f\|^{\,2+\frac2\zeta}
\end{multline*}
where $1/\zeta=\max\left\{2/d,\beta/k\right\}$, \emph{i.e.}, $\zeta=\min\left\{d/2,k/\beta\right\}$. Collecting terms, we have
\[
\dt \sfH[f]\le-\,C\,\zeta\,\sfH[f]^{1+\frac1\zeta}
\]
using~\eqref{H-norm},~\eqref{E-EP} and Proposition~\ref{prop:entropy_ineq}, with
\[
C:=\frac\kappa\zeta\,\min\big\{C_0\,z_0^{\frac\beta k-\frac1\zeta},\,C_1\,z_0^{\frac2d-\frac1\zeta}\big\}\(\tfrac2{1+\delta}\)^{1+\frac1\zeta}\,.
\]
Then the result of Theorem~\ref{th:main} follows from a Gr\"onwall estimate.
\[
\sfH[f(t,\cdot,\cdot)]\le\sfH[f^\init]\(1+C\,\sfH[f^\init]^\frac1\zeta\,t\)^{-\zeta}
\]
The expression of $C$ can be explicitly computed in terms of $C_0\,z_0^{\frac\beta k-\frac1\zeta}\,\sfH[f^\init]^\frac1\zeta$, which is proportional to $(\nrm{f^\init}{}/\nrm{f^\init}k)^\frac{2\,\beta}k$, and in terms of $C_1\,z_0^{\frac2d-\frac1\zeta}\,\sfH[f^\init]^\frac1\zeta$ which is a function of $(\nrm{f^\init}{\mathrm L^1(\d x\dd v)}/\nrm{f^\init}{})^{4/(d+2)}$, but it is of no practical interest. To see this, one has to take into account the expressions of $C_0$, $C_1$ and $\mathsf c$ in terms of the initial datum $f^\init$.
\qed

As a concluding remark, we emphasize that a control of the solution in the space $\mathrm L^2(\bangle v^k\dd x\ddmu)$, based on Proposition~\ref{prop:propag_L2_m}, is enough to prove Theorem~\ref{th:main}. In particular, there is no need of a uniform bound on $f$. This observation is new in $\mathrm L^2$ hypo\-coercive methods, and consistent with the homogeneous case (see Appendix~\ref{Appendix:B}).

\appendix\section{Weighted Poincar\'e inequalities}\label{Appendix:A}

This appendix is devoted to a proof of~\eqref{weightedPoincare} and considerations on related Poincar\'e inequalities. Inequality~\eqref{weightedPoincare} is not a standard weighted Poincar\'e inequality because the average in the right-hand side of the inequality involves the measure of the left-hand side so that the right-hand side cannot be interpreted as a variance. Here we prove a generalization of~\eqref{weightedPoincare} which relies on a purely spectral approach.

\subsection{Continuous spectrum and weighted Poincar\'e inequalities}\label{Sec:wPoincare}

Let us consider two probability measures on $\R^d$
\[
\d\xi=e^{-\phi}\dd v\quad\mbox{and}\quad \dnu=\psi\,\d\xi\,,
\]
where $\phi$ and $\psi\ge0$ are two measurable functions, and the \emph{weighted Poincar\'e inequality}
\be{wPoincare}
\forall\,h\in\mathcal D(\R^d)\,,\quad\intw{|\nabla h|^2}\xi\ge\C_\star\intw{\left|h-\widehat h\right|^2}\nu
\ee
where $\widehat h=\intw h\nu$. The question we address here is: \emph{on which conditions on $\phi$ and $\psi$ do we know that~\eqref{wPoincare} holds for some constant $\C_\star>0$ ?} Our key example is
\be{Example}
\phi(v)=\bangle v^\gam+\log Z_\gam\quad\mbox{and}\quad\psi(v)=c_{\gam,\beta}^{-1}\,\bangle v^{-\beta}
\ee
with $\gam>0$, $\beta>0$, $Z_\gam=\intw{e^{-\phi}}v$ and $c_{\gam,\beta}=\intw{\bangle v^{-\beta}}\xi$.

Let us consider a potential $\Phi$ on $\R^d$ and assume that it is a measurable function with
\[
\sigma=\lim_{r\to+\infty}\inf_{w\in\mathcal D(B_r^c)\setminus\{0\}}\frac{\intw{\(|\nabla w|^2+\Phi\,|w|^2\)}v}{\intw{|w|^2}v}>0\,,
\]
where $B_r^c:=\left\{v\in\R^d\,:\,|v|>r\right\}$ and $\mathcal D(B_r^c)$ denotes the space of smooth functions on~$\R^d$ with compact support in $B_r^c$. According to Persson's result~\cite[Theorem~2.1]{MR0133586}, the lower end~$\sigma_\star$ of the continuous spectrum of the Schr\"odinger operator $-\,\Delta+\Phi$ is such that
\[
\sigma_\star\ge\sigma\ge\lim_{r\to+\infty}\mathop{\mathrm{infess}}_{v\in B_r^c}\,\Phi(v)\,.
\]
If we replace $\intw{|w|^2}v$ by the weighted integral $\intw{|w|^2\,\psi}v$ for some measurable function $\psi$, we have the modified result that the operator $\mathcal L=\psi^{-1}\(-\,\Delta+\Phi\)$ on $\mathrm L^2(\R^d,\psi\dd v)$, associated with the quadratic form
\[
w\mapsto\intw{\(|\nabla w|^2+\Phi\,|w|^2\)}v
\]
has only discrete eigenvalues in the interval $(-\infty,\sigma)$ where
\[
\sigma=\lim_{r\to+\infty}\inf_{w\in\mathcal D(B_r^c)\setminus\{0\}}\frac{\intw{\(|\nabla w|^2+\Phi\,|w|^2\)}v}{\intw{|w|^2\,\psi}v}>0\,.
\]
To prove it, it is enough to observe that $0$ is the lower end of the continuous spectrum of $\mathcal L-\sigma_\star$ and to apply again~\cite[Theorem~2.1]{MR0133586}. It is also straightforward to check that the lower end of the continuous spectrum of $\mathcal L$ is such that
\[
\sigma_\star\ge\lim_{r\to+\infty}\mathsf q(r)=:\sigma_0\quad\mbox{where}\quad\mathsf q(r):=\mathop{\mathrm{infess}}_{B_r^c}\,\frac\Phi\psi\,.
\]
Notice that $\sigma_0$ is either finite or infinite. In the case of~\eqref{Example}, we get that $\sigma_0\in(0,+\infty]$ if and only if $\beta\ge 2\(1-\gam\)$. Relating the weighted Poincar\'e inequality~\eqref{wPoincare} with the spectrum of $\mathcal L$ is then classical. Let
\be{Schroedinger}
h=w\,e^{\phi/2}\,,\quad\Phi=\tfrac14\,|\nabla \phi|^2-\tfrac12\,\Delta \phi
\ee
and observe that
\begin{align*}
&\intw{|\nabla h|^2}\xi=Z_\gam^{-1}\intw{\(|\nabla w|^2+\Phi\,|w|^2\)}v\,,\\
&\intw{\left|h-\widehat h\right|^2}\xi=Z_\gam^{-1}\intw{\left|w-\widetilde w\right|^2\,\psi}v\,,
\end{align*}
where $\widetilde w=\frac{\intw{w\,\psi\,e^{-\phi/2}}v}{\intw{\psi\,e^{-\phi}}v}\,e^{-\phi/2}$.
\begin{proposition}\label{Prop:wPoincare} With the above notations, let $\Phi$ and $\psi$ be two measurable functions such that $\sigma_0>0$. Then inequality~\eqref{wPoincare} holds for some positive, finite, optimal constant $\C_\star>0$. Otherwise, if we have that $\lim_{r\to+\infty}\mathop{\mathrm{supess}}_{v\in B_r^c}\,\frac{\Phi(v)}{\psi(v)}=0$, then inequality~\eqref{wPoincare} does not hold.\end{proposition}

\begin{proof} By construction, $\sigma$ is nonnegative and the infimum of the Rayleigh quotient
\[
w\mapsto\frac{\intw{\(|\nabla w|^2+\Phi\,|w|^2\)}v}{\intw{|w|^2\,\psi}v}
\]
is achieved by $h\equiv\widehat h=1$, that is, by $w=\widetilde w=e^{-\phi/2}$, which moreover generates the kernel of $\mathcal L$. Hence we can interpret $\C_\star$ as the first positive eigenvalue, if there is any in the interval $(0,\sigma_\star)$, or $\C_\star=\sigma_\star$ if there is none.\end{proof}

In the case of~\eqref{Example}, the condition $\beta\ge2\,(1-\gam)$ is a necessary and sufficient condition for the inequality~\eqref{wPoincare} to hold. The threshold case $\beta=2\,(1-\gam)$ is remarkable: inequality~\eqref{wPoincare} can be rewritten for any $\gam\in(0,1)$ as the following \emph{weighted Poincar\'e inequality} :
\be{wPoincareAlpha}
\forall\,h\in\mathcal D(\R^d)\,,\quad\intw{|\nabla h|^2\,e^{-\bangle v^\gam}}v\ge\C_\star\intw{\frac{\left|h-\widehat h\right|^2\,e^{-\bangle v^\gam}}{\(1+|v|^2\)^{1-\gam}}}v\,,
\ee
for some constant $\C_\star\in\big(0,\gam^2/4\big]$ and
\[
\widehat h:=\frac1{z_\gam}\intw{\frac{h\,e^{-\bangle v^\gam}}{\(1+|v|^2\)^{1-\gam}}}v\,,\quad z_\gam=\intw{\frac{e^{-\bangle v^\gam}}{\(1+|v|^2\)^{1-\gam}}}v\,.
\]

\subsection{A weighted Poincar\'e inequality with a non-classical average}

\begin{corollary}\label{Cor:wPoincare} Let $\Phi$ and $\psi$ be respectively a measurable function and a bounded positive function such that, with the notations of section~\ref{Sec:wPoincare}, $\sigma_0>0$ and $\psi^{-1}\in\mathrm L^1(\R^d,\d\xi)$. Then the inequality
\be{wPoincare2}
\forall\,h\in\mathcal D(\R^d)\,,\quad\intw{|\nabla h|^2}\xi\ge\C\intw{\left|h-\widetilde h\right|^2}\nu
\ee
holds for some optimal constant $\C\in(0,\C_\star]$, where $\widetilde h:=\intw h\xi$. Here $\C_\star$ denotes the optimal constant in~\eqref{wPoincare}.\end{corollary}
As we shall see in the proof, our method provides us with an explicit lower bound $\C$ in terms of $\C_\star$. We emphasize that in~\eqref{wPoincare2}, the right-hand side is not the \emph{variance} of~$h$ with respect of the measure $\dnu$ because we subtract the average with respect to the measure $\d\xi$ which appears in the left-hand side. In case $\phi(v)=\bangle v^\gam$, inequality~\eqref{wPoincare2} is equivalent to~\cite[inequality~(1.12)]{hal-01241680}, which can be deduced using the strategy of~\cite{MR2386063,MR2381160}. Also see Appendix~\ref{Appendix:wP} for further details.
\begin{proof} Let us consider a function $h$. With no loss of generality, we can assume that $\widetilde h=\intw h\xi=0$ up to the replacement of $h$ by $h-\widetilde h$. We use the IMS decomposition method (see~\cite{MR526292,MR708966}), which goes as follows. Let $\chi$ be a truncation function on $\R_+$ with the following properties: $0\le\chi\le1$, $\chi\equiv1$ on $[0,1]$, $\chi\equiv0$ on $[2,+\infty)$ and ${\chi'}^2/\(1-\chi^2\)\le\kappa$ for some $\kappa>0$. Next, we define $\chi_R(v)=\chi\big(|v|/R\big)$, $h_{1,R}=h\,\chi_R$ and $h_{2,R}=h\,\sqrt{1-\chi_R^2}$, so that $h_{1,R}$ is supported in the ball $B_{2R}$ of radius $2R$ centered at $v=0$ and $h_{2,R}$ is supported in $B_R^c=\R^d\setminus B_R$. Elementary computations show that $h^2=h_{1,R}^2+h_{2,R}^2$ and $|\nabla h|^2=|\nabla h_{1,R}|^2+|\nabla h_{2,R}|^2-h^2\,|\nabla \chi|^2/\(1-\chi^2\)$, so that $\big||\nabla h|^2-|\nabla h_{1,R}|^2-|\nabla h_{2,R}|^2\big|\le\kappa\,h^2/R^2$.

Since $h_{2,R}$ is supported in $B_R^c$, we know that
\[
\intw{|\nabla h_{2,R}|^2}\xi\ge\mathsf q(R)\intw{|h_{2,R}|^2}\nu
\]
for any $R>0$, where $\mathsf q$ is the quotient involved in the definition of $\sigma_0$. We recall that $\lim_{r\to+\infty}\mathsf q(r)=\sigma_0>0$. Using the method of the Holley-Stroock lemma (see~\cite{holley1987logarithmic} and~\cite{doi:10.1142/S0218202518500574} for a recent presentation), we deduce from~inequality~\eqref{wPoincare} that
\begin{align*}
\intw{|\nabla h_{1,R}|^2}\xi&\ge\C_\star\intw{\left|h_{1,R}-\widehat h_{1,R}\right|^2}\nu\\
&\ge\C_\star\int_{B_{2R}}\left|h_{1,R}-\widehat h_{1,R}\right|^2\,\psi\dd\xi\\
&\ge\C_\star\,\inf_{B_{2R}}\psi\,\min_{c\in\R}\int_{B_{2R}}\left|h_{1,R}-c\right|^2\dd\xi\\
&\ge\mathsf Q(R)\intw{\left|h_{1,R}\right|^2}\nu-\C_\star\,\frac{\inf_{B_{2R}}\psi}{\xi(B_{2R})}\(\intw{h_{1,R}}\xi\)^2
\end{align*}
where $\mathsf Q(R):=\C_\star\,\inf_{B_{2R}}\psi/\sup_{B_{2R}}\psi$. By the assumption $\widetilde h=0$, we know that
\[
\int_{B_R}h\dd\xi=-\int_{B_R^c}h\dd\xi\,,
\]
from which we deduce that
\[
\(\intw{h_{1,R}}\xi\)^2=\(\int_{B_R}h\dd\xi+\int_{B_R^c}\chi\,h\dd\xi\)^2\le\(\int_{B_R^c}|h|\dd\xi\)^2\le\intw{h^2}\nu\int_{B_R^c}\psi^{-1}\dd\xi
\]
where the last inequality is simply a Cauchy-Schwarz inequality. Let
\[
\varepsilon(R):=\C_\star\,\frac{\inf_{B_{2R}}\psi}{\xi(B_{2R})}\int_{B_R^c}\psi^{-1}\dd\xi\,.
\]
By the assumption that $\psi^{-1}\in\mathrm L^1(\R^d,\d\xi)$, we know that
\[
\lim_{R\to+\infty}\varepsilon(R)=0\quad\mbox{and}\quad\lim_{R\to+\infty}\frac{\varepsilon(R)}{\mathsf Q(R)}=0\,.
\]
Collecting all our assumptions, we have
\begin{align*}
\intw{|\nabla h|^2}\xi &\ge\intw{\(|\nabla h_{1,R}|^2+|\nabla h_{2,R}|^2-\frac\kappa{R^2}\,h^2\)}\xi
\\
&\ge \(\min\big\{\mathsf Q(R),\mathsf q(R)\big\}-\varepsilon(R)-\frac\kappa{R^2}\)\intw{|h|^2}\nu
\end{align*}
where $\min\big\{\mathsf Q(R),\mathsf q(R)\big\}-\varepsilon(R)-\kappa/R^2$ is positive for $R>0$, large enough.

Finally, let us notice that for any $c\in\R$ we have
\[
\intw{\left|h-c\,\right|^2}\nu=\intw{h^2}\nu-2\,c\intw h\nu+c^2\ge\intw{\left|h-\widehat h\right|^2}\nu
\]
with equality if and only if $c=\widehat h=\intw h\nu$. As a special case corresponding to $c=\widetilde h=\intw h\xi$, we have
\[
\intw{\left|h-\widetilde h\right|^2}\nu\ge\intw{\left|h-\widehat h\right|^2}\nu.
\]
This proves that $\C_\star\ge\C$.\end{proof}

In the special case of~\eqref{Example}, it is possible to give a slightly shorter proof using the Poincar\'e inequality on $B_R$, for the measure $\d\xi$: see~\cite[Chapter~6]{Mischler-book}. An independent proof of such an inequality is anyway needed for a general $\phi$. The proof of Corollary~\ref{Cor:wPoincare} is more general and reduces everything to a comparison of the asymptotic behavior of~$\phi$ and $\psi$. If these functions are given by~\eqref{Example}, inequality~\eqref{wPoincare2} can be rewritten in the form of~\eqref{weightedPoincare}, we have an estimate of $\C$ and we can characterize $\C_\star$ as follows.
\begin{proposition} The optimal constant $\C_\star$ is the ground state energy of the operator $\mathcal L=\psi^{-1}\(-\,\Delta+\Phi\)$ on $\mathrm L^2(\R^d,\psi\dd v)$.\end{proposition}
The proof relies on~\eqref{Schroedinger}. Details are left to the reader.

\section{Algebraic decay rates for the Fokker-Planck equation}\label{Appendix:B}

Here we consider simple estimates of the decay rates in the homogeneous case given by $f(t,x,v)=g(t,v)$ of equation~\eqref{eq:main}, that is, the Fokker-Planck equation
\be{eqn:FP}
\partial_tg=\sfL_1 g\,.
\ee
After summarizing the standard approach based on the \emph{weak Poincar\'e inequality} (see for instance~\cite{hal-01241680}) in Section~\ref{Appendix:wP}, we introduce a new method which relies on \emph{weighted $\mathrm L^2$ estimates}. As already mentioned, the advantage of weighted Poincar\'e inequalities is that the description of the convergence rates to the local equilibrium does not require extra regularity assumptions to cover the transition from super-exponential ($\gam>1$) and exponential ($\gam=1$) local equilibria to sub-exponential local equilibria, with $\gam\in(0,1)$.

\subsection{Weak Poincar\'e inequality}\label{Appendix:wP}

We assume $\gam\in (0,1)$ and $\eta\in\big(0,\beta\big)$ with $\beta = 2\(1-\gam\)$. By a simple H\"older inequality, with $(\tau+1)/\tau=\beta/\eta$, we obtain that
\begin{multline*}
\intw{\left|h-\widetilde h\right|^2}\xi=\intw{\left|h-\widetilde h\right|^2\,\bangle v^{-\eta}\,\bangle v^\eta}\xi\\
\le\(\intw{\left|h-\widetilde h\right|^2\,\bangle v^{-\beta}}\xi\)^\frac{\tau}{\tau+1} \(\intw{\left\|h-\widetilde h\right\|_{\mathrm L^\infty(\R^d)}^2\,\bangle v^{\beta\,\tau}}\xi\)^\frac 1{1+\tau}\,.
\end{multline*}
Here we choose $\widetilde h:=\intw h\xi$. Using~\eqref{weightedPoincare}, we end up with
\begin{equation}\label{weightedInterp}
\forall\,h\in \mathcal D(\R^d),\quad\intw{\left|h-\widetilde h\right|^2}\xi\le\C_{\gam,\tau}\(\intw{|\nabla h|^2}\xi\)^\frac \tau{1+\tau}\,\left\|h-\widetilde h\right\|_{\mathrm L^\infty(\R^d)}^\frac{2}{1+\tau}\,,
\end{equation}
for some explicit positive constant $\C_{\gam,\tau}$. We learn from~\eqref{L1-diss} that
\[
\dt\intw{\left|h(t,\cdot)-\widetilde h\right|^2}\xi=-\,2\intw{|\Dv h|^2}\xi
\]
if $g=h\,\M$ is solves~\eqref{eqn:FP}, and we also know that $\widetilde h$ does not depend on $t$. By a strategy that goes back at least to~\cite[Theorem~2.2]{MR1112402} and is due, according to the author, to D.~Stroock, we obtain that
\[
\intw{\left|h(t,\cdot)-\widetilde h\right|^2}\xi\le\(\(\intw{\left|h(0,\cdot)-\widetilde h\right|^2}\xi\)^{-\frac{1}{\tau}}+\frac{2\,\tau^{-1}}{\C_{\gam,\tau}^{1+1/\tau}\,\mathcal M}\,t\)^{-\tau}
\]
with $\mathcal M=\sup_{s\in(0,t)}\left\|h(s,\cdot)-\widetilde h\right\|_{\mathrm L^\infty(\R^d)}^{2/\tau}$. The limitation is of course that we need to restrict the initial conditions in order to have $\mathcal M$ uniformly bounded with respect to~$t$. Since $\eta$ can be chosen arbitrarily close to $\beta$, the exponent $\tau$ can be taken arbitrarily large but to the price of a constant $\C_{\gam,\tau}$ which explodes as $\eta\to \beta_-$.

Notice that~\eqref{weightedInterp} is equivalent to the \emph{weak Poincar\'e inequality}
\[
\forall\,h\in\mathcal D(\R^d)\,,\quad\C_{\gam,\tau}^{-1}\intw{\left|h-\widetilde h\right|^2}\xi\le \tfrac\tau{(1+\tau)^{1+\frac1\tau}}\,r^{-\frac\tau{1+\tau}}\intw{|\nabla h|^2}\xi+r\,\left\|h-\widetilde h\right\|_{\mathrm L^\infty(\R^d)}^2\,,
\]
for all $r>0$, as stated in~\cite[(1.6) and example 1.4~(c)]{MR1856277}. The equivalence of this inequality and~\eqref{weightedInterp} is easily recovered by optimizing on $r>0$. It is worth to remark that here we consider $\left\|h-\widetilde h\right\|_{\mathrm L^\infty(\R^d)}$ while various other quantities like, \emph{e.g.}, the median can be used in weak Poincar\'e inequalities.

\subsection{Weighted \texorpdfstring{$\mathrm L^2$}{L2} estimates}\label{Appendix:Moments} As an alternative approach to the \emph{weak Poincar\'e inequality} method of Appendix~\ref{Appendix:wP}, we can consider for some arbitrary $k>0$ the evolution according to equation~\eqref{eqn:FP} of $\rint|h(t,v)|^2\,\bangle v^k\,\d\xi=\rint|h(t,v)|^2\,\bangle v^k\,\M\dd v$ where $h:=g/\M$ solves
\[
\partial_th=\M^{-1}\,\Dv\cdot\big(\M\,\Dv h\big)\,.
\]
Let us compute
\[
\dt\rint|h(t,v)|^2\,\bangle v^k\,\M\dd v+\,2\rint|\Dv h|^2\,\bangle v^k\,\M\dd v=-\rint\Dv(h^2)\cdot\big(\Dv\bangle v^k\big)\,\M\dd v
\]
and observe with $\ell=2-\gam$ that
\[
\Dv\cdot\(\M\,\Dv\bangle v^k\)=\frac k{\bangle v^4}\(d+(k+d-2)\,|v|^2-\gam\,\bangle v^\gam\,|v|^2\)\le a-b\,\bangle v^{-\ell}
\]
for some $a\in\R$, $b\in(0,+\infty)$. The same proof as in Proposition~\ref{prop:propag_L2_m} shows that there exists a constant $\mathcal K_k>0$ such that
\[
\forall\,t\ge0\,\quad\nrm{h(t,\cdot)}{\mathrm L^2\(\bangle v^k\,\d\xi\)}\le\mathcal K_k\,\nrm{h^\init}{\mathrm L^2\(\bangle v^k\dd\xi\)}\,.
\]
Hence, if $h$ solves~\eqref{eqn:FP} with initial value $h^\init$, we can use~\eqref{weightedPoincare} to write
\[
\dt\intw{\left|h(t,\cdot)-\widetilde h\right|^2}\xi=-\,2\intw{|\Dv h|^2}\xi\le-\,2\,\C\intw{\left|h-\widetilde h\right|^2\,\bangle v^{-\beta}}\xi
\]
with $\beta = 2\(1-\gam\)$ and $\widetilde h=\intw h\xi$. With $\theta=k/\big(k+\beta\big)$, H\"older's inequality
\[
\intw{\left|h-\widetilde h\right|^2}\xi\le\(\intw{\left|h-\widetilde h\right|^2\,\bangle v^{-\beta}}\xi\)^\theta\(\intw{\left|h-\widetilde h\right|^2\,\bangle v^k}\xi\)^{1-\theta}
\]
allows us to estimate the right hand side and obtain the following result.
\begin{proposition}\label{Prop:RateRelaxHom} Let $\gam\in(0,1)$, let $g^\init\in\mathrm L^1_+(\dmu)\cap\mathrm L^2(\bangle v^k\dmu)$ for some $k>0$, and consider the solution $g$ to~\eqref{eqn:FP} with initial datum $g^\init$. With $\C$ as in~\eqref{weightedPoincare}, if $\overline g=\(\rint g\dd v\)\M$ where~$\M$ is given by~\eqref{SubExponential}, then
\[
\intw{\left|g(t,\cdot)-\overline g\right|^2}\mu\le\(\(\intw{\left|g^\init-\overline g\right|^2}\mu\)^{-\beta/k}+\frac{2\,\beta\,\C}{k\,\mathcal K^{\beta/k}}\,t\)^{-k/\beta}
\]
with $\beta=2\,(1-\gam)$ and $\mathcal K:=\mathcal K_k^2\,\nrm{g^\init}{\mathrm L^2\(\bangle v^k\ddmu\)}^2+\Theta_k\(\rint g^\init\dd v\)^2$.\end{proposition}
We recall that $g=h\,\M$, $\,\overline g=\widetilde h\,\M$ and $\M\,\dmu = \d v = \M^{-1}\,\d\xi$. We notice that arbitrarily large decay rates can be obtained under the condition that $k>0$ is large enough. We recover that when $k<d\,\beta/2$, the rate of relaxation to the equilibrium is slower than $(1+t)^{-d/2}$ and responsible for the limitation that appears in Theorem~\ref{th:main}. However, the rate of the heat flow is recovered in Theorem~\ref{th:main} for a weight of order $k$ with an arbitrarily small but fixed $k>0$, if $\gam$ is taken close enough to $1$.
\begin{proof} Using
\[
\frac12\intw{\left|h-\widetilde h\right|^2\,\bangle v^k}\xi\le\intw{|h|^2\,\bangle v^k}\xi+\Theta_k\,\widetilde h^2=\mathcal K\,,
\]
we obtain that $y(t):=\intw{\left|g(t,\cdot)-\overline g\right|^2}\mu$ obeys to $y'\le-\,2\,\C\,\mathcal K^{1-1/\theta}\,y^{1/\theta}$ and conclude by a Gr\"onwall estimate.
\end{proof}

\newpage\section*{Acknowledgments} \begin{spacing}{0.9}{\small This work has been partially supported by the Project EFI (E.B., J.D., ANR-17-CE40-0030) of the French National Research Agency (ANR). Support by the Austrian Science Foundation (grants no.~F65 and W1245) is acknowledged by C.S. Moreover C.S. is very grateful for the hospitality at Universit\'e Paris-Dauphine. E.B., J.D., L.L.~and C.S.~are participants to the Amadeus project \emph{Hypocoercivity} no.~39453PH. All authors thank Cl\'ement Mouhot for stimulating discussions and encouragements.}\\
{\sl\scriptsize\copyright~2019 by the authors. This paper may be reproduced, in its entirety, for non-commercial purposes.}\end{spacing}

\bibliographystyle{acm}
\bibliography{BDLS-short}
\parindent=0pt\parskip=0pt\bigskip
\begin{center}
\rule{2cm}{0.5pt}
\end{center}
\end{document}